\documentclass[12pt]{amsart}
\usepackage[margin=1in]{geometry}                
\usepackage{graphicx}
\usepackage{comment}

\usepackage{amsmath}
\usepackage{amssymb}
\usepackage{amsfonts}
\usepackage{amsthm}
\usepackage{mathrsfs}
\usepackage{enumerate}
\usepackage{float}
\usepackage{fancyhdr}
\usepackage{xcolor}
\usepackage{hyperref}
\usepackage{resmes}
\usepackage{tikz}

\newtheorem{theorem}{Theorem}[section]

\newtheorem{lemma}[theorem]{Lemma}
\newtheorem{corollary}[theorem]{Corollary}

\theoremstyle{definition}
\newtheorem{definition}[theorem]{Definition}

\newtheorem{remark}[theorem]{Remark}

\numberwithin{equation}{section}

\DeclareMathOperator{\vspan}{Span}

\newcommand{\R}{\mathbb{R}}

\DeclareMathOperator{\Aff}{Aff}

\begin{document}
%
%
\title[Applications of Nonlinear Projections]{Applications of Nonlinear Projections\\to Rectifiable 1-sets}

\author[R.~Bongers]{Rosemarie Bongers}
\address{Department of Applied Mathematics \\ University of California, Merced, Merced, CA}
\email{rosemariebongers@ucmerced.edu}

\author[P.~Bright]{Paige Bright}
\address{Department of Mathematics \\ Massachusetts Institute of Technology, Cambridge, MA}
\email{paigeb@mit.edu}

\author[C.~Marshall]{Caleb Marshall}
\address{Department of Mathematics \\ The University of Toronto, Toronto, ON, Canada}
\email{caleb.marshall@utoronto.ca}

\author[K.~Taylor]{Krystal Taylor}
\address{Department of Mathematics, The Ohio State University, Columbus, OH}
\email{taylor.2952@osu.edu}

\date{\today}

\begin{abstract}
Projection theorems in Euclidean space provide a fundamental link between the geometric structure of a set and the size of its lower-dimensional images. For $1$-rectifiable sets in $\mathbb{R}^d$, a classical theorem of Federer shows that the $1$-dimensional Hausdorff measure of such sets is controlled by the multiplicity-weighted lengths of finitely many linearly independent projections. We develop a framework for extending Federer's result into a diverse set of nonlinear problems. This technique yields a unified approach for studying sets through their lower-dimensional nonlinear images, as well as studying the exceptional sets which exhibit poor projective behavior.

As illustrations of our technique, we show that (i) every 1-rectifiable set contains a pin whose pinned distance set has positive Lebesgue measure, and that the exceptional set of pins for which this fails is contained in a $(d-2)$-dimensional affine subspace; (ii) planar radial projections of a 1-rectifiable set can fail to have positive length from at most one vantage point unless the set is essentially linear; and finally (iii) unions of circles centered on a 1-rectifiable set have positive area under mild assumptions on the radius function. 
\end{abstract}

\maketitle


\section{Introduction}

A foundational theme in geometric measure theory is that the geometry of a set is reflected in the size of its projections.
Typically, results are stated in terms of a set's Hausdorff dimension in a given range. At the critical dimension, one typically requires some regularity condition on the set. The most important smoothness condition is rectifiability; heuristically, rectifiable sets behave like Lipschitz manifolds. More precisely, let $d \geq 2$ and $\mathcal{H}^1$ denote the $1$-dimensional Hausdorff measure on $\mathbb{R}^d$. We say a Borel set $E \subseteq \mathbb{R}^d$ is a $1$-\textbf{rectifiable set} if there are Lipschitz mappings $f_j : \mathbb{R} \rightarrow \mathbb{R}^d$ with
$$
\mathcal{H}^1 \bigg(  E \setminus \bigcup_{j = 1}^{\infty} f_j (\mathbb{R})\bigg) = 0.
$$
Conversely, a Borel set $E \subseteq \mathbb{R}^d$ is said to be a \textbf{purely $1$-unrectifiable set} if
$$
\mathcal{H}^1 \big(E \cap f (\mathbb{R})\big) = 0 \textrm{ for every Lipschitz function } f : \mathbb{R} \rightarrow \mathbb{R}^d.
$$
Any Borel set $E \subseteq \mathbb{R}^d$  satisfying $0 < \mathcal{H}^1 (E) < \infty$ (which we refer to as \textbf{$1$-sets}), admit a pairwise-disjoint decomposition
\begin{equation}\label{eq:1setdecomposition}
E : = E_{r} \sqcup E_{u} \sqcup E_{0}
\end{equation}
where $E_r$ is a $1$-rectifiable set, $E_u$ is purely $1$-unrectifable set, and $E_0$ is an $\mathcal{H}^1$-null set. Moreover, after modifying the exact choice of decomposition, we may assume that both the sets $E_r,E_u \subseteq \mathbb{R}^d$ are Borel sets.

\subsection{Projections of 1-Rectifiable Sets}

A fundamental feature of 1-rectifiable sets in $\R^d$ is that, among any $d$ projections onto linearly independent lines, at least one must have a positive length image. 
This is a classical theorem of Federer \cite[Theorem 3.2.27]{Federer1996}, itself a complement of the famous Besicovitch--Federer projection theorem. We refer to this result as Federer's projection theorem, a version of which is stated below.

\begin{theorem}[Federer's projection theorem]\label{thm:federerprojection-intro}
Let $E \subseteq \mathbb{R}^d$ be any $1$-rectifiable set, and let $e_1,\ldots,e_d$ denote any orthonormal basis of $\mathbb{R}^d$. Writing $x = \sum_i x_i e_i$, let
$$
\pi_j(x) := x_j \in \mathbb{R}, \qquad \forall j \in \{1,\ldots,d\}.
$$
Defining
$$
a_j := \int_{\pi_j(E)} \mathcal{H}^0\big(\pi_j^{-1}(y) \cap E\big)\,d\mathcal{H}^1(y),
$$
one then has the inequality
\begin{equation}\label{eq:federerthminequality-intro}
\big(a_1^2 + \cdots + a_d^2 \big)^{1/2} \lesssim \mathcal{H}^1\big(E \big) \lesssim a_1 + \cdots + a_d.
\end{equation}
where the implied constant depends on the orthonormal basis and dimension, but not $E$.
\end{theorem}

Federer's projection theorem bounds the length of $E \subseteq \mathbb{R}^d$ by the lengths of its projections $\pi_1(E)$, $\ldots$, $\pi_d(E)$, with each point $y \in \pi_j(E)$ weighted by the corresponding multiplicity of its fibre's intersection with $E$. If $E$ is a segment in the plane, then \eqref{eq:federerthminequality-intro} is nothing more than the standard triangle inequality. Despite the geometric simplicity of this description of Federer's result for line segments, the result depends on weak differential-geometric properties of $1$-rectifiable sets, such as the local density of $\mathcal{H}^1$ and the existence of weak tangents almost everywhere. Moreover, the upper bound of \eqref{eq:federerthminequality-intro} can fail if $E$ is a purely $1$-unrectifiable set of positive and finite length.

To observe this failure, take any collection of self-similar Cantor sets $\mathcal{C}_1,\ldots,\mathcal{C}_d$, each satisfying the open set condition and the dimensional conditions
$$
\min \{\dim_H \mathcal{C}_1,\ldots,\dim_H \mathcal{C}_d \} > 0, \quad \sum\limits_{j = 1}^{d} \dim_H \mathcal{C}_j = 1
$$
and define
$$
\mathcal{S}=\mathcal{S}(\mathcal{C}_1,\ldots,\mathcal{C}_d) := \mathcal{C}_1 \times \cdots \times \mathcal{C}_d.
$$
Here and throughout, $\dim_H$ refers to Hausdorff dimension. As a consequence of the open set condition (see \cite[Chapter 4.13]{Mattila1995} for a standard reference), one has that $\mathcal{H}^1 (\mathcal{S}) > 0
$
and yet, for any $j \in \{1,\ldots,d\}$, we have
$$
\mathcal{H}^1 \big(\pi_{j} (\mathcal{S})\big) = \mathcal{H}^1 \big(\mathcal{C}_{j} \big) = 0,
$$
since each $\mathcal{C}_j$ has Hausdorff dimension strictly less than one. Such product Cantor sets $\mathcal{S}$ are standard examples of \emph{purely $1$-unrectifiable sets}. These form a natural counterpoint to rectifiable sets; while rectifiable sets have a natural Lipschitz manifold structure that allows for one to employ many classical analytic tools (such as variants of Fubini's theorem), unrectifiable sets must typically be studied in an \emph{ad hoc} manner; see for instance \cite{NPV2010} for a canonical examination of Favard length problem in the plane, Marshall's work in \cite{Marshall2025} for the higher dimensional variant of the problem, and \cite{Laba2015} for a standard expository reference for the Favard length problem. As such, it is worthwhile to keep these examples of unrectifiable sets in mind when considering the following theorem of Besicovitch and Federer.

\begin{theorem}[The Besicovitch--Federer projection theorem]\label{thm:besifed}
Let $d \geq 2$ and $E \subseteq \mathbb{R}^d$ be any $1$-set. Then, the following conditions are equivalent.
\begin{enumerate}
    \item The set $E$ is a purely $1$-unrectifiable set.
    \smallskip
    \item One has $\mathcal{H}^1 (\pi_{\theta} (E)) =0$ for $\sigma^{d-1}$ almost-every $\theta \in \mathbb{S}^{d-1}$. Here, $\sigma^{d-1}$ denotes the standard surface measure on the sphere and $\pi_\theta : \R^d \to \R$ denotes orthogonal projection onto the linear subspace $\ell_\theta$ in direction $\theta$.
\end{enumerate}
\end{theorem}

As a consequence of Federer's projection theorem and the Besicovitch--Federer projection theorem, one can easily obtain the following $d$-lines projection result (see Section \ref{sec:overview}).

\begin{corollary}[Federer's $d$-lines projection theorem]\label{thm:federerdlines}
    Suppose that $E \subseteq \mathbb{R}^d$ is a $1$-set with decomposition $E = E_r \sqcup E_u \sqcup E_0$ as in \eqref{eq:1setdecomposition}. The following conditions are equivalent.
    \begin{enumerate}[(1)]
        \item The $1$-rectifiable part of $E$ satisfies $\mathcal{H}^1 (E_r) > 0$.
        \medskip
        \item For any list of directions $\theta_1,...,\theta_d \in \mathbb{S}^{d-1}$ satisfying $\vspan (\theta_1,...,\theta_d) = \mathbb{R}^d$, there necessarily exists some $j^* \in \{1,...,d\}$ such that $\mathcal{H}^1\big(\pi_{\theta_{j^*}} (E)\big) > 0$.
\end{enumerate}
\end{corollary}

In this paper, we develop a technique for using Federer's projection theorem in an array of nonlinear projection problems in geometric measure theory. Our framework has numerous applications for problems which do not immediately exhibit a projective nature: in the following sections, we study pinned distance sets, radial projections, and the areas of unions of circles. Our technique allows us to develop a broad framework for showing that the lower-dimensional images of a rectifiable set under a broad class of nonlinear maps must be large. Moreover, we obtain structural results for exceptional sets where the nonlinear projections are atypically small. We postpone background of these topics to the corresponding sections so that we may first give an overview of our technique and main results. 

\subsection{Overview of our technique and main results} \label{sec:overview}

The classical Besicovitch--Federer projection theorem applies in one particular context: we have a family of $d$ orthogonal projections into a one-dimensional space. We have both a structural result (that a rectifiable set has positive length projections in almost every direction) and a more precise combinatorial result for sets which are not purely unrectifiable (that given any set of $d$ linearly independent directions, at least one of the images will have positive length). The key utility of this result is its replacement of a high-dimensional problem ($1$-rectifiability in $\mathbb{R}^d$) with a collection of low-dimensional problems (positive length in $\mathbb{R}$). The main result of this paper is a technique which substantially generalizes the kinds of maps which can be used to detect rectifiability, as well as to derive better combinatorial information from them. 

At a high level, the technique involves four steps:

\begin{enumerate}[(1)]
\setcounter{enumi}{0}
    \item Fix a domain $\Omega \subseteq\mathbb{R}^d$ together with an indexed family $\{\varphi_{\alpha_1}, \dots, \varphi_{\alpha_d}\}$ of maps $\Omega \to \mathbb{R}$. Form the \emph{canonical embedding} 
    $$\mathsf{H}_{\vec \alpha}(z) = \left(\varphi_{\alpha_1}(z), \dots, \varphi_{\alpha_d}(z)\right)$$
    with $\mathsf{H}_{\vec \alpha} : \Omega \to \mathbb{R}^d.$
    \item Prove that for any sufficiently nice choice of indices $\vec \alpha$, the associated $\mathsf{H}_{\vec \alpha}$ preserves positivity of measure, in the sense that
    $$\mathcal{H}^1(E) > 0 \implies \mathcal{H}^1(\mathsf{H}_{\vec \alpha}(E)) > 0$$
    for all sets $E \subseteq \Omega$. This part of the argument relies upon local bilipschitzness, a co-Lipschitz lower bound, or another explicit positivity-preserving condition adapted to the particular family of curve projections. On the localized portion under consideration, this also ensures that $\mathsf{H}_{\boldsymbol{\alpha}}(E)$ is rectifiable whenever $E$ is.
    \item There is a natural correspondence between the orthogonal projections $\pi_{j} \circ \mathsf{H}_{\vec \alpha}$ and the nonlinear projections $\varphi_{\alpha_j}$. Use the linear version of Federer's projection theorem to conclude that, if $E$ is $1$-rectifiable, then there is a ``good'' index $j$ for which $$\mathcal{H}^1(\varphi_{\alpha_j}(E)) > 0.$$
    \item Extract combinatorial information from the previous step: ``any'' collection of $d$ indices $\{\alpha_1, \dots, \alpha_d\}$ will include a good index. Use this to prove structural results regarding the exceptional set of ``bad'' indices $\alpha$ for which $\mathcal{H}^1(\varphi_{\alpha}(E)) = 0$. 
\end{enumerate}

The utility of this technique is two-fold. By organizing the work so that Federer's projection theorem (i.e. the coarea formula used in step (3), Theorem \ref{thm:federerprojection-intro}) is treated as a black box, the technical difficulty is relocated to step (2) and merely requires showing that the associated canonical embedding $\mathsf{H}_{\vec \alpha}$ preserves positivity of length. In many contexts, this can be done via ordinary calculus or by identifying that the embedding is bilipshictz. Moreover, the combinatorial information in step (4) can be used to prove structural results regarding the exceptional set of ``bad'' indices for which $\mathcal{H}^1(\varphi_{\alpha_j}(E)) = 0$.

As a first illustration of the technique, we begin with a proof of the $d$-lines theorem using this language. If we choose $\varphi_{j}$ to be orthogonal projection onto $\ell_{\theta_j}$, then the associated canonical embedding $\mathsf{H}_{\vec \alpha}$ in step (1) is a linear automorphism which clearly preserves positivity of length; the combinatorial information in step (4) is precisely the $d$-lines theorem. 

\begin{proof}[Proof of Corollary \ref{thm:federerdlines}] 
Fix a $1$-rectifiable set $E \subseteq \mathbb{R}^d$ with positive length and a family of directions $\{\theta_1, \dots, \theta_d\}$ satisfying $\mathrm{Span}(\theta_1, \dots, \theta_d) = \R^d$. Let $\varphi_j$ denote orthogonal projection onto $\ell_{\theta_j}$ and observe that the associated canonical embedding $\mathsf{H}(z) = (\varphi_1(z), \dots, \varphi_d(z))$ is a linear automorphism on $\mathbb{R}^d$. It immediately follows that $\mathcal{H}^1(\mathsf{H}(E)) > 0$. By Federer's projection theorem,
$$0 < \mathcal{H}^1(\mathsf{H}(E)) \lesssim \sum_{j = 1}^d \int_{\mathbb{R}} \mathcal{H}^0\big(\pi_j^{-1}(y) \cap \mathsf{H}(E)\big) \, d\mathcal{H}^1(y).$$
Choose an index $j^{\ast}$ for which the selected summand is positive; there is a set $T \subseteq \mathbb{R}$ with positive length such that
$$t \in T \implies \mathcal{H}^0\big(\pi_{j^{\ast}}^{-1}(t) \cap \mathsf{H}(E)\big) > 0 \implies \pi_{j^{\ast}}^{-1}(t) \cap \mathsf{H}(E) \ne \emptyset.$$
There is a relationship between the intersection of the fibers $\pi_{j^{\ast}}^{-1}(\cdot)$ with $\mathsf{H}(E)$ and the corresponding intersection of the fibers $\varphi_{j^{\ast}}^{-1}(\cdot)$ with $E$:
\begin{align*}
    z \in \pi_{j^{\ast}}^{-1}(t) \cap \mathsf{H}(E) &\implies \pi_{j^{\ast}}(z) = t \text{ and } z \in \mathsf{H}(E) \\
    &\implies t \in \pi_{j^{\ast}}(\mathsf{H}(E)) \\
    &\implies t \in \varphi_{j^{\ast}}(E) \\
    &\implies \varphi_{j^{\ast}}^{-1}(t) \cap E \ne \emptyset.
\end{align*}
Therefore, whenever $t \in T$ we have that $\varphi_{j^{\ast}}^{-1}(t) \cap E \ne \emptyset.$ Integrating over $t$ shows that $\mathcal{H}^1(\varphi_{j^{\ast}}(E)) > 0$; this is the desired result, recalling that $\varphi_{j^{\ast}} = \pi_{\theta_{j^{\ast}}}.$

For the other direction of the equivalence, note that if $E$ is a set for which any list of directions $\{\theta_1, ..., \theta_d\} \subseteq \mathbb{S}^{d - 1}$ spanning $\mathbb{R}^d$ has a ``good index'' $j^{\ast}$ (i.e. for which $\mathcal{H}^1\big(\pi_{\theta_j^*}(E)\big) > 0$), then the corresponding set of ``bad directions'' (i.e. those $\theta$ for which $\mathcal{H}^1\big(\pi_{\theta}(E)\big) = 0$) must have $\sigma^{d - 1}$ measure zero. By Theorem \ref{thm:besifed}, $E$ is not purely $1$-unrectifiable and so $\mathcal{H}^1(E_r) > 0$. 
\end{proof}

 In the following sections, we apply our technique in a number of different contexts, especially to problems that do not immediately appear projective in nature.

\begin{itemize}
\item In Section 2, we consider the maps $\varphi_{p}(z) = |z - p|^2$ indexed by points $p \in \mathbb{R}^d$. The workflow enables us to prove that for any $1$-rectifiable set $E \subseteq \mathbb{R}^d$ which is not essentially $1$-flat, the pinned distance set
$$\Delta_p(E) = \{|z - p| : z \in E\}$$
has positive length for Lebesgue-almost every pin $p \in \mathbb{R}^d$. Moreover, a more careful consideration of the combinatorial information will show that the exceptional set of bad pins must actually be contained in a $(d-2)$-dimensional affine subspace of $\mathbb{R}^d$. 
\item In Section 3, we take $\varphi_{p}$ to be an arctangent function relative to the point $p \in \mathbb{R}^2$ and use the workflow to prove that for a $1$-rectifiable set $E \subseteq \mathbb{R}^2$, the radial projection $\pi_{p}(E)$ from the vantage point $p \in \mathbb{R}^2$ will have positive length for almost every $q$. The resulting structural statement is that the collection of ``bad'' vantage points from which $\pi_p(E)$ has zero length must in fact be contained in a straight line. 
\item In Section 4, we use the variable-radius version of a curve-based projection map appearing in the work of Simon and Taylor \cite{ST2022}; see also Li and Taylor \cite{LT2026}. We prove that if $E \subseteq \mathbb{R}^2$ is $1$-rectifiable, then the collection of variable-radii circles
$$\bigcup_{z \in E} C(z, r(z))$$
has positive area given a very mild regularity assumption on the radius function $r$.
\end{itemize}
\noindent It is worth noting that recent work of Li and Taylor obtains a quantitative 2-lines theorem for nonlinear projections in $\R^2$. As a consequence, their work gives alternative proofs of our main pinned distance set and radial projection results established in Sections \ref{sec:pinneddistance} and \ref{sec:radproj}; see \cite[Corollary 2.8 and Theorem 2.10]{LT2026}.

\subsection{Acknowledgments}
This work grew out of collaborations fostered by the American Institute of Mathematics through the SQuaREs project \textit{Covering Fractals by Curves}. P.B. was supported by a \textit{MathWorks Fellowship} at MIT. C.M. is supported by an \textit{Arts and Science Postdoctoral Fellowship} at the University of Toronto, as well as a \textit{Canadian Postdoctoral Fellowship Award} from the National Science and Engineering Research Council of Canada. K.T. is supported in part by Simons Foundation Grant GR137264.

\section{Pinned distance sets}\label{sec:pinneddistance}

Given a point $p \in \mathbb{R}^d$ and a set $E\subseteq \mathbb{R}^d$, the \emph{pinned distance set} of $E$ at $p$ is 
$$\Delta_p(E) = \{|z - p| : z \in E\} \subseteq \R.$$
We refer to a point $p$ as a \emph{good pin} if $\mathcal H^1(\Delta_p(E)) > 0$ and a \emph{bad pin} otherwise. We will let $\mathcal{B}$ denote the set of bad pins, allowing the dependence on $E$ to be implicit.

The classical Falconer distance set conjecture asks how large a Borel set $E\subseteq\mathbb{R}^d$ must be in order that its unpinned distance set $\{|x-y|:x,y\in E\}$ have positive length. Falconer \cite{Falconer1985} showed that the threshold cannot be below $d/2$ and conjectured that $\dim_H(E)>d/2$ should suffice. A stronger pinned conclusion asks for a point $p\in E$ such that $\mathcal H^1(\Delta_p(E)) >0$. The conjectured threshold remains open for every $d\geq2$. 

It is not immediately obvious that this is a projective problem; however, significant progress has been made using projection theory. The best results in this direction are due to Guth--Iosevich--Ou--Wang and Du--Ou--Ren--Zhang, who collectively prove that for all compact sets $E\subseteq \R^d$ with $\dim_H(E)> f(d)$, there exists a $p\in E$ such that $\mathcal H^1(\Delta_p(E))>0$, where
\[
f(d) = \begin{cases}
    \frac{5}{4}, & d=2 \quad \quad \text{(Guth--Iosevich--Ou--Wang \cite{GIOW2020})} \\
    \frac{d}{2} + \frac{1}{4} - \frac{1}{8d + 4}, & d\geq 3 \quad \quad \text{(Du--Ou--Ren--Zhang \cite{DORZ2023})}.
\end{cases}
\]
Both of these results make use of radial projections, the former making use of work of Orponen's projection theorem \cite{Orponen2019} and the latter using Ren's \cite{Ren2023}. See Section \ref{sec:radproj} for more background on radial projections. 

Utilizing rectifiability, and in particular Federer's projection theorem, we obtain a pinned distance set result for 1-rectifiable sets that are not essentially $k$-flat. 

\begin{definition}
    A set $E \subseteq \R^d$ is \emph{essentially $k$-flat} if there exists a $k$-plane $P$ such that $\mathcal H^1(E \setminus P) = 0$.
\end{definition}

\begin{remark}
Note that the above is a slight variant of Orponen and Sahlsten's definition of essentially $k$-flat \cite{OS2011}, in which $0< \mathcal H^k(E)$ and $\mathcal H^k(E\setminus P) = 0$ for some $k$-plane $P$. On the other hand, in the radial projection section, our definition of $k$-flat will completely match with Orponen--Sahlsten as our attention will be focused to the plane.
\end{remark}

We begin with a qualitative result; as indicated by the general workflow above, this will be followed with structural results extracted combinatorially from the qualitative theorem.
\begin{theorem}\label{thm:main_pinned_result}
Let $E \subseteq \mathbb{R}^d$ be a 1-rectifiable set with $\mathcal{H}^1(E) > 0$ and $\mathcal{P} = \{p_1, p_2, \dots, p_d\} \subseteq \mathbb{R}^d$ be a set of $d$ pins. If $\mathcal{P}$ is affinely independent and $E$ is not essentially $(d-1)$-flat, then $\mathcal{P}$ contains 
at least one pin $p$ for which $\mathcal{H}^1(\Delta_p(E))>0.$
\end{theorem}
As indicated in the framework above, the primary tool in establishing this result is a regularity result for the canonical embedding corresponding to appropriately-chosen projection maps $\varphi : \mathbb{R}^d \to \mathbb{R}$. To this end, we have the following lemma:
\begin{lemma}\label{lemma:pinned_distance_canonical_embedding} For each point $p \in \mathbb{R}^d$, define $\varphi_p(z) = |z - p|^2$. Given a collection $\mathcal{P} = \{p_1, p_2, \cdots, p_d\} \subseteq \mathbb{R}^d$ such that $\Aff \mathcal{P}  = \mathbb{R}^{d - 1} \times \{0\}$, form the associated canonical embedding
$$\mathsf{H}_{\mathcal{P}}(z) = \big(\varphi_{p_1}(z), \dots, \varphi_{p_d}(z)\big).$$
Finally, given any $\epsilon > 0$, let $\Omega_{\epsilon} = \mathbb{R}^{d - 1} \times (\epsilon, 1/\epsilon).$ Under these assumptions, the Jacobian determinant $\det D\mathsf{H}_{\mathcal{P}} : \Omega_{\epsilon} \to \mathbb{R}$ is bounded away from zero and infinity.
\end{lemma}

\begin{proof}
The argument proceeds by a direct calculation of the Jacobian determinant. To this end, note that
$$\nabla \varphi_{p_j} = 2 (z - p_j).$$
Regarding each point in $\mathbb{R}^d$ as a $d \times 1$ column vector, it follows that
$$D\mathsf{H}_{\mathcal{P}} = 2 \begin{bmatrix} | & | & \dots & | \\ z - p_1 & z - p_2 & \dots & z - p_d \\
| & | & \dots & |\end{bmatrix}.$$
Note that subtracting the first column from any other column does not change the value of the determinant of this matrix; carrying this out in each column yields
$$\det(D\mathsf{H}_{\mathcal{P}}) = 2^d \det \begin{bmatrix}
    | & | & \dots & | \\ 
    z - p_1 & z - p_2 & \dots & z - p_d \\
    | & | & \dots & |
\end{bmatrix}.$$
Each $p_i - p_j$ is contained in $\mathbb{R}^{d - 1} \times \{0\}$; the final coordinate $\pi_d(z)$ of $z$ is between $\epsilon$ and $1/\epsilon$. Our matrix therefore has a useful block structure, where a natural abuse of notation allows us to regard $q_i - q_j$ as $(d - 1)$-vectors:
$$\det(D\mathsf{H}_{\mathcal{P}}) = 2^d \det \begin{bmatrix}
    \pi_{\mathbb{R}^{d - 1}}(z - q_1) & \begin{bmatrix} | & \dots & | \\ p_1 - p_2 & \dots & p_1 - p_d \\ | & \dots & |\end{bmatrix} \\ \pi_d(z) & - \quad \vec 0 \quad - 
\end{bmatrix}.$$
By cofactor expansion along the final row, we have
$$\det(D\mathsf{H}_{\mathcal P}) = 2^d \pi_d(z) \det
\begin{bmatrix} | & \dots & | \\ p_1 - p_2 & \dots & p_1 - p_d \\ | & \dots & |\end{bmatrix}.
$$
Since the points $\{p_1, p_2, \dots, p_d\}$ are affinely independent, this final determinant is non-zero. Moreover, by our choice of domain $\Omega_{\epsilon}$, $\pi_d(z)$ is between $\epsilon$ and $1/\epsilon$. Since the Jacobian determinant of $\mathsf{H}_{\mathcal{P}}$ is bounded away from both zero and infinity, the result follows. 
\end{proof}

For our results on pinned distance sets, it would have been more natural to use $\varphi_p(z) = |z - p|$ rather than the squared version; however, that would have made the derivative calculations substantially more challenging. This will not be an obstacle to using these particular maps, however; it will be addressed in the proof of the main theorem of this section, which we are now ready to address.

\begin{proof}[Proof of Theorem \ref{thm:main_pinned_result}]
Assume that $\mathcal{P}$ is affinely independent and that $E$ is not essentially $(d - 1)$-flat. Let $V_{\mathcal{P}} = \Aff \mathcal{P}$ and $V_{\mathcal{P}}(\epsilon)$ be its corresponding $\epsilon$-neighborhood in $\mathbb{R}^d$. By the dominated convergence theorem, if $\epsilon$ is sufficiently small, we have
$$\mathcal{H}^1(E \setminus V_{\mathcal P}(\epsilon)) > 0.$$
Following a rotation, we may therefore assume without loss of generality that $V_{\mathcal{P}} = \mathbb{R}^{d - 1} \times \{0\}$ and that $$E \subseteq \Omega_{\epsilon} = \mathbb{R}^{d - 1} \times (\epsilon, 1/\epsilon).$$
Form the canonical embedding $\mathsf{H}_{\mathcal{P}}$ as before; by Lemma \ref{lemma:pinned_distance_canonical_embedding}, its Jacobian determinant is bounded below; furthermore, the function $\mathsf{H}_{\mathcal{P}}$ is locally Lipschitz due to the upper bound on $\det D\mathsf{H}_{\mathcal{P}}$. Therefore $\mathsf{H}_{\mathcal{P}}(E)$ is a $1$-rectifiable set in $\mathbb{R}^d$ with positive length. By Federer's projection theorem,
$$0 < \mathcal{H}^1(\mathsf{H}_{\mathcal P}(E)) \lesssim \sum_{j = 1}^d \int_{\mathbb{R}} \mathcal{H}^0\big(\pi_j^{-1}(t) \cap \mathsf{H}_{\mathcal{P}}(E)\big) d\mathcal{H}^1(t),$$
where $\pi_j$ is orthogonal projection onto the $j$-th coordinate axis in $\mathbb{R}^d$. Select an index $j$ for which
$$\int_{\mathbb{R}} \mathcal{H}^0\big(\pi_j^{-1}(t) \cap \mathsf{H}_{\mathcal{P}}(E)\big) d\mathcal{H}^1(t) > 0.$$
Fix any $t$ for which $\mathcal{H}^0\big(\pi_j^{-1}(t) \cap \mathsf{H}_{\mathcal{P}}(E)\big) > 0$; it is worth noting that $t > 0$ because $\mathsf{H}_{\mathcal P}(E)$ is contained in the first orthant of $\mathbb{R}^d$. Since $\mathcal{H}^0$ is the counting measure, there is a point $w \in \pi_j^{-1}(t) \cap \mathsf{H}_{\mathcal{P}}(E)$ and thus a $z \in E$ with $\mathsf{H}_{\mathcal{P}}(z) = w$. Since $\pi_j$ simply gives the $j$-th coordinate of a point, the definition of $\mathsf{H}_{\mathcal{P}}$ gives us that $\varphi_{p_j}(z) = t.$ Therefore, $\varphi_{p_j}^{-1}(t) \cap E \ne \emptyset$. In short,
$$\mathcal{H}^0\big(\pi_j^{-1}(t) \cap \mathsf{H}_{\mathcal P}(E)\big) > 0 \implies \mathcal{H}^0(\varphi_j^{-1}(t) \cap E\big) > 0.$$
Therefore, the support of $t \mapsto \mathcal{H}^0\big(\varphi_{p_j}^{-1}(t) \cap E\big)$ contains the support of $t \mapsto \mathcal{H}^0\big(\pi_j^{-1}(t) \cap \mathsf{H}_{\mathcal{P}}(E)\big)$; as the latter as positive length, so does the former.

By the above reasoning, there is an index $j$ and a set $T_j \subseteq [0, \infty)$ with positive length for which 
$$t \in T_j \implies \varphi_{p_j}^{-1}(t) \cap E \ne \emptyset.$$
Interpreting this in terms of pinned distance sets, we have that 
$$t \in T_j \implies t^{1/2} \in \Delta_{p_j}(E).$$
Since $T_j \subseteq [0, \infty)$ has positive length, so does $\{t^{1/2} : t \in T_j\}.$ Therefore, $\mathcal H^1(\Delta_{p_j}(E)) > 0$ and $p_j$ is a good pin, as desired.
\end{proof}

We next establish the accompanying structural results.
The key challenge here is that a general $1$-rectifiable subset of $\mathbb{R}^d$ may be contained in a low-dimensional affine subspace, and so does not immediately meet the non-flatness assumption of Theorem \ref{thm:main_pinned_result}. Observe that $E$ is essentially $d$-flat but not essentially $0$-flat, since it is contained in $\mathbb{R}^d$ and has more than one point. As such, we may choose a dimension $1 \le k \le d$ such that $E$ is essentially $k$-flat but not essentially $(k - 1)$-flat. In this case, we may assume without loss of generality that $E \subseteq \mathbb{R}^k \times \{0\}^{d - k}$; alternatively, replacing $E$ with its orthogonal projection into $\mathbb{R}^k$ allows us to regard $E$ as a subset of $\mathbb{R}^k$ itself. Our first structural result is on the existence of a good pin within $E$.

\begin{theorem}\label{thm:pinned_distance_structural_good}
If $E \subseteq \mathbb{R}^d$ is a $1$-rectifiable set with positive length, then $E$ contains a good pin.
\end{theorem}

\begin{proof}
If $d = 1$, the result is immediate: every pin in $\mathbb{R}$ is a good pin. Otherwise let $d \ge 2$ and choose the $k \ge 1$ for which $E$ is essentially $k$-flat but not essentially $(k - 1)$-flat and regard $E$ as a subset of $\mathbb{R}^k$. Note that such a $k$ exists as $E$ has positive length. Choose any collection $\mathcal{Q}$ of $k$ points from $E$ which are affinely independent. Since $E$ is not essentially $(k - 1)$ flat, Theorem \ref{thm:main_pinned_result} implies that at least one point in $\mathcal{Q}$ is a good pin as desired.
\end{proof}

There is a corresponding structural result for the set of bad pins: it must be rather low-dimensional. We will need to consider multiple cases based on the geometry of $E$. In the case that $E$ is contained in a line, the result will follow from a geometric reduction and a short calculation. When $E$ exhibits less flatness, we will need to extract combinatorial information from Theorem \ref{thm:main_pinned_result} and use the geometry of distance sets. 

\begin{theorem}\label{thm:pinned_distance_structural_bad}
    Let $E \subseteq \mathbb{R}^d$ be a $1$-rectifiable set with positive length. The set $\mathcal{B}$ of bad pins for $E$ is contained in an at most $(d-2)$-dimensional subspace of $\mathbb{R}^d$.
\end{theorem} 

\begin{proof}
First consider the case that $E$ is essentially $1$-flat; without loss of generality, we may assume that $E$ is contained within the first coordinate axis $\ell_1$ of $\mathbb{R}^d$. Fix a point $p \in \mathbb{R}^d$; we wish to show that $p$ is a good pin. By translating, we may assume that the first coordinate $p_1 = 0$. Split $E$ into
$$E_+ = \{e \in E : e_1 \ge 0\}$$
and its corresponding $E_-$, where $e_1$ again stands for the first coordinate of the point $e \in \mathbb{R}^d$. Since $E_+ \cup E_- = E$, we may assume without loss of generality that $E_+$ has positive measure.

The origin is a good pin for $E$, because
$$t \in \Delta_0(E_+) \iff (t, 0, \dots, 0) \in E_+$$
and therefore $\mathcal H^1(\Delta_0(E))\geq \mathcal H^1(\Delta_0(E_+)) > 0$. Moreover, if $p$ is any other point with $p_1 = 0$, there is a correspondence
$$t \in \Delta_0(E_+) \iff \sqrt{t^2 + |p|^2} \in \Delta_p(E_+).$$
Since $|p| > 0$, the function $t \mapsto \sqrt{t^2 + |p|^2}$ is locally bilipschitz and therefore $p$ is also a good pin. We may conclude that every point in $\mathbb{R}^d$ is a good pin and indeed that $\mathcal{B} = \emptyset$.

Now consider the higher-dimensional setting, where $E$ is essentially $k$-flat but not essentially $(k - 1)$-flat. Following a rotation and translation, we may assume as before that $E \subseteq \mathbb{R}^{k} \times \{0\}^{d - k}$. Let $\pi : \mathbb{R}^d \to \mathbb{R}^k$ denote the orthogonal projection onto the first $k$ coordinates, $\widetilde{E} = \pi(E)$, and $\widetilde{\mathcal{B}}$ is the set of bad pins for $\widetilde{E}$ in $\mathbb{R}^k$. More precisely,
\[
\widetilde{B} = \{q \in \R^k : \mathcal H^1\big(\Delta_q(\pi(E))\big) = 0 \}.
\]
If $\widetilde{\mathcal{B}}$ were not contained within a $(k - 2)$-dimensional affine subspace of $\mathbb{R}^k$, we could extract an affinely independent set of $k$ bad pins from $\widetilde{B}$, contradicting the result of Theorem \ref{thm:main_pinned_result}. Therefore, $\widetilde{B}$ is contained within a $(k - 2)$-dimensional affine subspace $V_{\widetilde{\mathcal B}}$ of $\mathbb{R}^k$. 

We now return to considering $E$ as a subset of the full-dimensional ambient space $\mathbb{R}^d$. We will show that $\mathcal{B} \subseteq \pi^{-1}(\widetilde{\mathcal{B}})$; this is sufficient to complete the proof, since this latter set is contained in $V_{\widetilde{\mathcal{B}}} \times \{0\}^{d - k}$ which is indeed $(k - 2) + (d - k) = (d-2)$-dimensional. To this end, choose a point $p \in \mathbb{R}^d$ such that $\pi(p) \notin \widetilde{\mathcal{B}}$ so that
$$\mathcal H^1(\Delta_{\pi(p)}(\widetilde E)) > 0.$$
If $\delta = \operatorname{dist}(p, \pi(\mathbb{R}^d))$, we have
$$t \in \Delta_{\pi(p)}(\widetilde E) \iff \sqrt{t^2 + \delta^2} \in \Delta_p(E).$$
Since $\delta > 0$, the same reasoning as above shows $p$ is a good pin. This completes the proof.
\end{proof}

A slightly stronger structural result can be concluded from the proof of the previous result. The bilipschitz equivalence of $\Delta_{\pi(p)}(\widetilde{E})$ and $\Delta_p(E)$ actually implies that $$B = \widetilde{B} \times \mathbb{R}^{d - k};$$ that is, the set of bad pins in the $d$-dimensional setting for set which is not essentially $k$-flat has a Cartesian product structure. Moreover, this result is sharp at the level of dimension, as can be seen by example. If $E \subseteq \mathbb{R}^d$ is the one-dimensional unit circle contained in the plane spanned by the first two coordinate axes, then $\mathcal{B} = \{0\}^2 \times \mathbb{R}^{d - 2}.$

\section{Radial projections}\label{sec:radproj}

We now consider radial projections in the plane. Given a \emph{vantage point} $p \in \mathbb{R}^2$, we consider the projection
$$\pi_p(z) = \frac{z - p}{|z - p|},$$
which can be regarded as a map from $\mathbb{R}^2 \setminus \{p\}$ to either $\mathbb{R}^{1}$ or $\mathbb{S}^1$. Given a set $E \subseteq \mathbb{R}^2$, we will refer to a vantage point as \emph{good} if $\mathcal{H}^1(\pi_p(E)) > 0$, and \emph{bad} otherwise. 

Radial projections have been a rapidly developing topic in the field of projection theory, with foundational results regarding rectifiability. For the purpose of discussion, and its pertinence to our planar result, we restrict our attention to $E\subseteq \mathbb{R}^2$ whose collection of bad vantage points is denoted $\mathcal B$. Building upon classical work of Marstrand \cite{Marstrand1954}, Mattila and Orponen \cite{MO2016,Orponen2017}  showed that given $\dim_H(E) >1$, almost every vantage point $p\in \R^2$ is good. The subcritical regime $\dim_H (E) < 1$ yields results for dimension rather than measure, and has seen significant progress over the last decade, see \cite{BFR2024,BG2024,CS2025,OSW2024}.

The critical case $\dim_H(E) = 1$ more subtly depends on the rectifiability of $E$.
Building upon the Besicovitch--Federer theorem (Theorem \ref{thm:besifed}), Marstrand \cite{Marstrand1987} showed that for purely 1-unrectifiable sets $E$,
\[
\dim (\R^2 \setminus \mathcal B) \leq 1.
\]
If $E$ is 1-rectifiable, we obtain a Besicovitch-type result recovering a result of Orponen--Sahlsten \cite{OS2011}. Besicovitch's classical result for orthogonal projections states that a rectifiable subset of $\mathbb{R}^2$ can have zero projection length in only a small number of directions; we obtain an analogous result for radial projections.

\begin{theorem}\label{thm:main_radial_projections}
Let $E \subseteq\mathbb{R}^2$ be a $1$-rectifiable set with positive length.
\begin{itemize}
    \item If $E$ is essentially $1$-flat, the set of bad vantage points for $E$ is a line.
    \item If $E$ is not essentially $1$-flat, there is at most one bad vantage point for $E$.
\end{itemize}
\end{theorem}

As in the prior section, our result will follow from proving that a certain map is bilipschitz; we will then be able to reduce matters to using Federer's projection theorem to show the existence of a rich set of good pins. The relevant projections are indeed the maps $\varphi_p$; some care is needed to avoid branch cuts, but this will be addressed as part of a geometric reduction in the proof of Theorem \ref{thm:main_radial_projections}.

\begin{lemma}\label{lemma:lemma:radial_projection_canonical_embedding}
Fix an $\epsilon > 0$ and let $\Omega_{\epsilon} = (-1/\epsilon, 1/\epsilon) \times (\epsilon, 1/\epsilon)$. Define two functions $\varphi_1$ and $\varphi_2$ by 
$$\varphi_j(z) = \arctan \left(\frac{y}{x \pm \frac 1 2}\right),$$
where $j = 1, 2$ correspond to the signs $\pm$, respectively and where the arctangent takes values in $(0, \pi)$. The canonical embedding $\mathsf{H}_{\pm}(z) = (\varphi_1(z), \varphi_2(z))$ is bilipschitz on $\Omega_{\epsilon}$.
\end{lemma}

The relevance of these maps is geometric. If we take $p = (1/2, 0)$ and $q = (-1/2, 0)$ to be two points in the real axis, $\mathsf{H}_\pm(z)$ returns the angles (measured from the positive $x$-axis) of a point $z$ in the upper half space from the vantage points $p$ and $q$, respectively.

\begin{proof}
    If $z = (x, y)$ and we let $x_{\pm} := x \pm \frac 1 2$, a direct calculation shows that
    $$\det D\mathsf{H}_{\pm}(z) = \frac{y}{(x_+^2 + y^2)(x_-^2 + y^2)}.$$
    This is bounded away from $0$ and $\infty$ on $\Omega_{\epsilon}$, with bilipschitz constants depending on $\epsilon$.
\end{proof}

We next prove the main theorem of this section; the main work that is required is a geometric reduction so that we can use the previous lemma. 

\begin{proof}[Proof of Theorem \ref{thm:main_radial_projections}]
Fix any two vantage points $p$ and $q$ in the plane; by translating, scaling, and rotating the set and vantage points, we may assume without loss of generality that $p = (-1/2, 0)$ and $q = (1/2, 0)$. By reflecting over the $x$-axis if needed, we may also assume that $\mathcal{H}^1(E \cap \mathbb{H}_+) > 0$, where $\mathbb{H}_+$ is the closed upper half plane. 

We first address the case that $E$ is not essentially $1$-flat. Since $\mathcal{H}^1(E \setminus (\mathbb{R} \times \{0\})) > 0$ the dominated convergence theorem implies the existence of an $\epsilon$ such that $\mathcal{H}^1(E \cap \Omega_{\epsilon}) > 0$; we may as well assume that $E \subseteq \Omega_{\epsilon}$. As usual, form the associated canonical embedding $\mathsf{H}_{\pm}$ as in Lemma \ref{lemma:lemma:radial_projection_canonical_embedding}; since this is bilipschitz, Federer's projection theorem implies that
$$0 < \mathcal{H}^1(\mathsf{H}_{\pm}(E)) \lesssim \sum_{j = 1}^2 \int_{\mathbb{R}} \mathcal{H}^0\big(\pi_j^{-1}(t) \cap \mathsf{H}_{\pm}(E)\big) \, d\mathcal{H}^1(t),$$
where $\pi_j$ is orthogonal projection onto the $j$-th coordinate axis in $\mathbb{R}^2$. As such, there is an index $j^* \in \{1, 2\}$ for which 
$$\int_{\mathbb{R}}  \mathcal{H}^0\big(\pi_{j^*}^{-1}(t) \cap \mathsf{H}_{\pm}(E)\big) \, d\mathcal{H}^1(t) > 0.$$

As in the proof of Theorem \ref{thm:main_pinned_result}, we now relate notation to translate from $\pi_{j^*}$ to $\varphi_{j^*}$. If $t \in \mathbb{R}$ is such that $ \mathcal{H}^0\big(\pi_{j^*}^{-1}(t) \cap \mathsf{H}_{\pm}(E)\big) > 0$, we have $\pi_{j^*}(t) \cap \mathsf{H}_{\pm}(E) \ne \emptyset$; that is, $t \in \varphi_{j^*}(E)$. Summarizing,
$$ \mathcal{H}^0\big(\pi_{j^*}^{-1}(t) \cap \mathsf{H}_{\pm}(E)\big) > 0 \implies  \mathcal{H}^0\big(\varphi_{j^*}^{-1}(t) \cap E\big) > 0.$$
Since $\int_{\mathbb{R}} \mathcal{H}^0\big(\pi_{j^*}^{-1}(t) \cap \mathsf{H}_{\pm}(E)\big) \, d\mathcal{H}^1(t) > 0$, there exists a set $T_{j^*} \subseteq \mathbb{R}$ with positive length for which 
$$t \in T \implies \mathcal{H}^0\big(\varphi_{j^*}^{-1}(t) \cap E\big) > 0 \implies t \in \varphi_{j^*}(E).$$
Therefore, $\mathcal{H}^1(\varphi_{j^*}(E)) > 0$. Recalling that $\varphi_j$ encodes the angles subtended by $E$ from the vantage points $p$ and $q$ respectively, we conclude that $\mathcal{H}^1(\pi_p(E)) > 0$ or $\mathcal{H}^1(\pi_q(E)) > 0.$ Therefore, at least one of $\{p, q\}$ is a good vantage point. Since $p$ and $q$ were arbitrary points in $\mathbb{R}^2$, we conclude that there is at most one bad vantage point in the plane. 

We now turn to the flat case, which is much simpler. Since $\mathcal{H}^1$-null sets will also have null radial projections, we may assume without loss of generality that $E \subseteq \mathbb{R}$ viewed as the $x$-axis in $\mathbb{R}^2$. If $p \in \mathbb{R}$, then $\pi_p(E)$ consists of at most two points and $p$ is a bad vantage point. If $p \notin \mathbb{R}$, one may verify that $\pi_p : \mathbb{R} \times \{0\} \to \mathbb{S}^1$ is bilipschitz on any bounded interval; from this, it follows that $p$ is a good vantage point.
\end{proof}

The non-flat conclusion is sharp. For example, let $E$ be the union of two non-collinear line segments meeting at the origin. Its radial image from the origin consists of two points, while Theorem \ref{thm:main_radial_projections} shows that every other vantage point is good. Thus $\mathcal{B}=\{0\}$.

\section{Unions of circles with centers in a rectifiable set}\label{sec:circles}
 Classical results of Marstrand and, independently, Bourgain show that if a set contains a circle centered at every point of a set of positive planar measure, then the set itself must have positive Lebesgue measure \cite{Bourgain1986, Marstrand1987}. 

It is now known that if $E\subseteq \R^2$ satisfies $\dim_H(E) >1$, 
and $r:\R^2\rightarrow (0,\infty)$ is a measurable selection, 
then the union of circles centered at points of $E$, 
$$\mathcal{C}= \bigcup_{v\in E}C(v,r(v))$$
has positive Lebesgue measure \cite{Wolff2000}; see also Mitsis \cite{Mitsis1999}. Further, if $s:= \dim_H(E)<1$, then 
$$\dim_H(\mathcal{C}) = 1+s,$$ see \cite{Oberlin2007}. 
Thus, in both the planar and higher-dimensional settings, Hausdorff dimension $1$ emerges as the critical threshold for positive measure unions of circles and spheres.

At the critical threshold $\dim_H(E)=1$, the situation is more subtle. Simon and Taylor \cite{ST2022} showed that, for circles of fixed radius, rectifiability rather than dimension alone determines whether the union has positive Lebesgue measure. More precisely, 
if $E$ is a $1$-set, then the union of circles centered at points of $E$ (of fixed radius) has positive Lebesgue measure if and only if $E$ is not purely unrectifiable. 


In the present paper, we show that a union of circles centered on a $1$-rectifiable subset of $\mathbb{R}^2$ has positive measure when the radius is allowed to vary, subject to a mild regularity condition. Before stating the regularity condition, note that some degree of regularity is necessary. A construction of Talagrand demonstrates the existence of a Lebesgue measure zero set containing a circle centered at every point of a line \cite{Talagrand1980}. 
In contrast to Talagrand's construction, our main result is that a mild regularity and monotonicity condition on the radius function is sufficient to guarantee that the union of circles centered on a rectifiable set has positive Lebesgue measure. 

\begin{definition}
    A function $r : \mathbb{R}^2 \to (0, \infty)$ is called an \emph{admissible radius function} if
    \begin{itemize}
        \item there exist real numbers $a$ and $b$ for which $0 < a \le r(z) \le b < \infty$ for all $z \in \mathbb{R}^2$,
        \item $r$ is differentiable with $\|\nabla r\|_\infty  < \infty$, and 
        \item the partial derivatives $r_x$ and $r_y$ do not change sign.
    \end{itemize}
\end{definition}

\noindent Our main result is the following.

\begin{theorem}\label{thm:main_circle_union_result}
Let $r$ be an admissible radius function and $E$ a $1$-rectifiable subset of $\mathbb{R}^2$ with positive $\mathcal{H}^1$-measure. If 
$$\mathcal{C} = \bigcup_{v \in E} C(v, r(v)),$$
then $\mathcal{C}$ has positive two-dimensional Lebesgue measure.
\end{theorem}

As in the previous sections, our result will follow by first identifying the relevant projection-type map, proving a regularity result for the associated canonical embedding, and then extracting combinatorial information. Given $\alpha\in \R$ and a point $z = (x, y)\in \R^2$, define
$$\varphi_{\alpha}(z) = y + \sqrt{r^2(z) - (\alpha - x)^2}.$$
Geometrically, $\varphi_{\alpha}(x, y)$ yields the $y$-coordinate of the intersection of the upper half of $C(z, r(z))$ and the vertical line $\ell_{\alpha} = \{\alpha\} \times \mathbb{R}$. The projective nature of the problem is now apparent: ignoring some issues regarding the domain of definition of $\varphi_{\alpha}$, we have that $\varphi_{\alpha}(E) = \{\alpha\} \times (\mathcal{C}_+ \cap \ell_{\alpha})$, where $\mathcal{C}_+$ refers to the union of the upper halves of the circles $C(v, r(v))$. Taking Lebesgue measure on both sides and integrating in $\alpha$ then shows a relationship between an ``average projection length'' and the area of $\mathcal{C}_+$. We now proceed with the details of the proof, beginning with a study of the associated canonical embedding.

\begin{lemma} \label{lemma:circles_canonical_embedding}
    Fix numbers $0 < a \le b < \infty$. Let $r : \mathbb{R}^2 \to (a, b)$ be a differentiable function for which $r_x \le 0$ and $r_y \ge 0$ and with bounded gradient. If $\alpha$ and $\beta$ are real numbers for which $\frac{a}{20} < \alpha - \beta < \frac{a}{10}$, then the associated canonical embedding $\mathsf{H}_{(\alpha, \beta)} = (\varphi_{\alpha}, \varphi_{\beta})$ is bilipschitz on the vertical strip with $\alpha < x < \alpha + \frac a {10}.$ The bilipschitz constant for $\mathsf{H}_{(\alpha, \beta)}$ depends on $a, b$, and $r$ but not on the selection of $\alpha$ and $\beta$.
\end{lemma}

\begin{proof}
    We begin with a calculation. Since $r$ is differentiable on its domain, the gradient of $\varphi_{\alpha}$ can be calculated as
    $$\nabla \varphi_{\alpha}(x, y) = \begin{bmatrix} \frac{r(x, y) r_x(x, y) + (\alpha - x)}{\sqrt{r^2(x, y) - (\alpha - x)^2}} \\ 1 + \frac{r(x, y) r_y(x, y)}{\sqrt{r^2(x, y) - (\alpha - x)^2}}\end{bmatrix} =: \begin{bmatrix} f_{\alpha}(x, y) \\ 1 + g_{\alpha}(x, y)\end{bmatrix}.$$
    Fix $\alpha < \beta$ and form the associated canonical embedding
    $$\mathsf{H}_{(\alpha, \beta)}(z) = (\varphi_{\alpha}(z), \varphi_{\beta}(z)).$$
    Using the gradient calculation, we find that
    $$\det D\mathsf{H}_{(\alpha, \beta)}(z) = \underbrace{f_{\alpha}(z) - f_{\beta}(z)}_{(I)} + \underbrace{\big(f_{\alpha}(z) g_{\beta}(z) - f_{\beta}(z) g_{\alpha}(z)\big)}_{(II)}.$$
    Our goal is to show that this is bounded above and below on the appropriate domain. Suppressing the arguments of the function $r$ for ease of notation, we can rewrite term (II) as 
    \begin{align*}
        (II) &= \frac{\big(r \cdot r_x + (\alpha - x)\big)\big(r\cdot r_y\big) - \big(r \cdot r_x + (\beta - x)\big)\big(r\cdot r_y\big)}{\sqrt{r^2 - (\alpha - x)^2}\sqrt{r^2 - (\beta - x)^2}} \\
        &= \frac{r \cdot r_y \cdot \big((\alpha - x) - (\beta - x)\big)}{\sqrt{r^2 - (\alpha - x)^2}\sqrt{r^2 - (\beta - x)^2}} \\
        &= \frac{r \cdot r_y \cdot (\alpha - \beta)}{\sqrt{r^2 - (\alpha - x)^2}\sqrt{r^2 - (\beta - x)^2}}.
    \end{align*}
    Likewise, term (I) can be rewritten as
    \begin{align*}
        (I) &= r \cdot r_x \left[\frac{1}{\sqrt{r^2 - (x - \alpha)^2}} - \frac{1}{\sqrt{r^2 - (x - \beta)^2}}\right] + \left[\frac{x - \beta}{\sqrt{r^2 - (x - \beta)^2}} - \frac{x - \alpha}{\sqrt{r^2 - (x - \alpha)^2}}\right].
    \end{align*}
    Combining (I) and (II) and refactoring, we have
    \begin{align*}
        \det D\mathsf{H}_{(\alpha, \beta)}(x, y) &= \left[\frac{(x - \beta) / r}{\sqrt{1 - ((x - \beta)/r)^2}} - \frac{(x - \alpha) / r}{\sqrt{1 - ((x - \alpha)/r)^2}} \right] \\
        &\quad\quad + r_x \left[\frac{1}{\sqrt{1 - ((x - \alpha) / r)^2}} - \frac{1}{\sqrt{1 - ((x - \beta)/r)^2}}\right] \\
        &\quad\quad + \frac{r_y}{r} \left[\frac{\alpha - \beta}{\sqrt{1 - ((x - \alpha)/r)^2} \sqrt{1 - ((x - \beta)/r)^2}}\right] \\
        &= (\text{A}) + (\text B) + (\text C).
    \end{align*}

For term (A), note that the function $\psi(t) = \frac{t}{\sqrt{1 - t^2}}$ is strictly increasing, and in fact $\psi'(t) \ge 1$. As a consequence, we have that
$$(\text{A}) \ge \frac{x - \beta}{r} - \frac{x - \alpha}{r} = \frac{\alpha - \beta}{r} \ge \frac{\alpha - \beta}{\|r\|_{\infty}} \ge \frac{a/20}{b} >  0.$$
Term (A) is also bounded above, since both $(x - \beta)/r$ and $(x - \alpha)/r$ lie in the interval $(0, 1/2)$. Note that these bounds depend on $r$ (and hence on $a$ and $b$), but not on $\alpha$ or $\beta$. 

Term (B) can be handled in a similar manner. Note that $\eta(t) = \frac{1}{\sqrt{1 - t^2}}$ is a strictly increasing function, so the bracketed term is negative (but as in Term (A), it is bounded in absolute value). Because $r_x \le 0$, it follows that $(\text B) \ge 0$ but is bounded above.

Term (C) can also be easily bounded: because $(x - \beta)/r$ and $(x - \alpha)/r$ both lie in $(0, 1/2)$, the bracketed portion is between $\alpha - \beta$ and $\frac 4 3 (\alpha - \beta)$. In any case, since $r_y/r$ is nonnegative and bounded above, the overall term is nonnegative and bounded. 

Combining the three upper and lower bounds, we conclude that the Jacobian determinant is bounded away from $0$ and $\infty$ with constants independent of $\alpha$ and $\beta$, as claimed. 
\end{proof}

We are now ready to prove the main theorem of this section.
\begin{proof}[Proof of Theorem \ref{thm:main_circle_union_result}] Suppose that $r$ is an admissible radius function with lower bound $r(z) \ge a > 0$. Since $r_x$ and $r_y$ do not change sign, and since area is preserved under reflections, we may assume without loss of generality that $r_x \le 0$ and $r_y \ge 0$ in order to apply the previous lemma. Furthermore, we may also assume that $E \subseteq (\frac a {10}, \frac{11a}{100}) \times \mathbb{R}$. 

Fix any $\beta \in (0, a/100)$ and $\alpha \in (a/11, a/10)$ and note that $a/20 < \alpha - \beta < a/10$. By Lemma \ref{lemma:circles_canonical_embedding}, the associated canonical embedding $\mathsf{H}_{(\alpha, \beta)}$ is bilipschitz on the vertical strip $(\alpha, \alpha + a/10) \times \mathbb{R}$; this strip contains $(a/10, 11a/100) \times \mathbb{R}$ for any such $\alpha$. By Federer's projection theorem,
$$0 < \mathcal{H}^1(\mathsf{H}_{(\alpha, \beta)}(E)) \lesssim \sum_{j = 1}^2 \int_{\mathbb{R}} \mathcal{H}^0\big(\pi_j^{-1}(t) \cap \mathsf{H}_{(\alpha, \beta)}(E)\big) d\mathcal{H}^1(t),$$
where $\pi_j$ is orthogonal projection onto the $j$-th coordinate axis of $\mathbb{R}^2$. As in the proof of Theorem \ref{thm:main_pinned_result}, we now unravel the relationship between the coordinates of the embedding and the orthogonal projection. We have that
$$
    \mathcal{H}^0\big(\pi_1^{-1}(t) \cap \mathsf{H}_{(\alpha, \beta)}(E)\big) > 0 \implies \pi_1^{-1}(t) \cap \mathsf{H}_{(\alpha, \beta)}(E) \ne \emptyset
    \implies t \in \varphi_{\alpha}(E).
$$
Similar reasoning relates $\pi_2$ and $\beta$. From Federer's projection theorem, we conclude that there exists a set $T \subseteq\mathbb{R}$ with positive $\mathcal{H}^1$ (or Lebesgue) measure for which $T \subseteq \varphi_{\alpha}(E)$ or $T \subseteq \varphi_{\beta}(E)$, depending on which orthogonal projection actually has positive integral. In any case, either $|\varphi_{\alpha}(E)| > 0$ or $|\varphi_{\beta}(E)| > 0$. 

We now extract the relevant combinatorial information and complete the proof. If there exists an $\alpha \in (a/11, a/10)$ for which $|\varphi_{\alpha}(E)| = 0$, then $|\varphi_{\beta}(E)| > 0$ for every $\beta \in (0, a/100)$. Alternatively, $|\varphi_{\alpha}(E)| > 0$ for every $\alpha \in (a/11, a/10)$. Since the union of circles is a measurable set, we may apply the Fubini--Tonelli theorem to conclude that 
$$\left|\bigcup_{z\in E} C(z,r(z))\right| \geq \int_0^{a/10} |\varphi_{\gamma}(E)| \, d\gamma > 0.$$
Since $\{\gamma\} \times \varphi_{\gamma}(E) \subseteq \mathcal{C}$ for each $\gamma$ in this interval, $\mathcal{C}$ has positive area.
\end{proof}

Theorem \ref{thm:main_circle_union_result} is closely related to a recent result of Iosevich, Li, and Taylor on unions of variable hypersurfaces. In \cite{ILT26}, it is proved that if $E$ is $1$-rectifiable with $\mathcal{H}^1(E)>0$, then the union of a family of hypersurfaces satisfying suitable curvature hypotheses has positive Lebesgue measure. In the planar setting considered here, the hypersurfaces are circles with centers in $E$ and radii prescribed by an admissible radius function.

Although both results establish positivity of measure at the rectifiable endpoint, the underlying mechanisms are quite different. The proof in \cite{ILT26} relies on smoothing estimates arising from the mapping properties of Fourier integral operators, whereas our argument is entirely geometric and is based on the nonlinear two-projection theorem.  In this way, Theorem \ref{thm:main_circle_union_result} may be viewed as a variable-radius circle analogue of the rectifiable endpoint theorem in \cite{ILT26}, established through a substantially different method.

We also note a complementary result of {\L}aba, McDonald, and Taylor \cite{LMT26}. They show that if the centers lie in a purely unrectifiable self-similar $1$-set, then the union of circles has Lebesgue measure zero whenever the radii vary sufficiently slowly. Taken together, these results highlight a sharp contrast between rectifiable and purely unrectifiable sets at the critical dimension $1$ and demonstrate that rectifiability plays a decisive role in determining the measure-theoretic size of unions of circles with variable radii.

\section*{Future Directions}
As observed in Section \ref{sec:overview}, our methods for studying pinned distances, radial projections and unions of circles all follow a similar strategy. In all three examples, our strategy centers around the stability of the canonical encoding map $\mathsf{H}_{\boldsymbol{\alpha}}$, and an application of Federer's projection theorem. In follow-up work, Marshall \cite{Marshall2026FiniteGoodWitnesses} explores this framework, proving similar good witness results (as well as new local stability results for good witnesses) for a wide class of generalized curve projections. We direct the reader to this forthcoming work for more.

\bibliography{refs}
\bibliographystyle{abbrv}
                
\end{document}